\documentclass[a4paper,12pt,reqno]{amsart}

\usepackage{amssymb,amsmath,amsfonts}

\usepackage[all]{xy}
\input xy

\theoremstyle{definition}
\newtheorem{definition}{Definition}
\newtheorem{proposition}{Proposition}

\newcounter{example}
\newcommand{\Cat}[1]{\ensuremath{{\bf #1}}}
\newcommand{\catD}{\ensuremath{\Cat{D}}}
\newcommand{\catDG}{\ensuremath{\Cat{D_G}}}
\newcommand{\catEG}{\ensuremath{Ext(\catDG)}}
\newcommand{\catS}{\ensuremath{\Cat{Sets}}}
\newcommand{\catC}{\ensuremath{\Cat{C}}}
\newcommand{\CtoS}{${\catS}^{{\catC}^{op}}$}
\newcommand{\EtoS}{${\catS}^{{Ext(\catDG)}^{op}}$}

\begin{document}
\title[Syntax topologies]{The topology of syntax relations of a formal language}
\author{Vladimir Lapshin}
\makeatletter
\newdimen\whitespaceamount
\settowidth\whitespaceamount{~}
\def\@@and{\hskip-\whitespaceamount,~}
\makeatother

\begin{abstract}
The method of constructing of Grothendieck's topology basing on a neighbourhood grammar, defined on the category of syntax diagrams is described in the article. Syntax diagrams of a formal language are the multigraphs with nodes, signed by symbols of the language's alphabet. The neighbourhood grammar allows to select correct syntax diagrams from the set of all syntax diagrams on the given alphabet by mapping an each correct diagram to the cover consisted of the grammar's neighbourhoods. Such the cover gives rise to Grothendieck's topology on category $\catEG$ of correct syntax diagrams extended by neighbourhoods' diagrams. An each object of category $\catEG$ may be mapped to the set of meanings (abstract senses) of this syntax construction. So, the contrvariant functor from category $\catEG$ to category of sets $\catS$ is defined. The given category \EtoS likes to be seen as the convenient means to think about relations between syntax and semantic of a formal language. The sheaves of set defined on category $\catEG$ are the objects of category \EtoS that satisfy of compositionality principle defined in the semantic analysis.
\end{abstract}
\maketitle

\section{Introduction}

The formal language's syntax traditionally is described by using of the notion of grammar. The grammar defined the laws that are the base to build correct syntax constructs from atomic entities (symbols). The method, described in \cite{Lapshin}, allows to universally describe the syntax of a formal language in spite of representation of its texts (linear one or not). The method describes the syntax constructs by using of the notion of syntax diagram. The syntax diagram is the connected multigraph with nodes signed by a formal language's alphabet and ribs can belong to different sorts and represent the syntax relations. The multigraph of a syntax diagram may be directed or not. The main restriction to the multigraphs of syntax diagrams is connectivity. It is possible to select correct syntax diagrams from the set of all syntax diagrams on the given alphabet. The formalism of neighbourhood grammar is used to do this. It may to define the set of subdiagrams for each syntax diagram $D$ as the set of pairs $(D',s)$, where $D'$ is a syntax diagram and $s$~-- inclusion mapping of syntax diagram $D'$ to syntax diagram $D$. The neighbourhood an alphabet's symbol is a syntax diagram, which contains the node, signed by this symbol. This node is named as the center of the neighbourhood. The neighbourhood grammar is a finite family of neighbourhoods defined for each symbol of the alphabet. The syntax diagram is named as correct one if for each its node signed by some symbol of the alphabet it includes some neighbourhood of this symbol. Such the neighbourhood should contain all ribs adjoining to its center, the set of such the ribs is named as the neighbourhood's star. So, there is at least one cover consisted of neighbourhoods for each correct syntax diagram in the given neighbourhood grammar. Such the cover is named as the syntax one. Further, the category $\catD$ of syntax diagrams above the given alphabet will be described. Also, it will be shown how, for the given category of syntax diagrams basing on neighbourhood grammar, define the category of correct syntax diagrams and Grothendieck's topology on it.

\section{Category $\catD$ of syntax diagrams}

Define category $\catD$ of syntax diagrams above the fixed alphabet $A$ and the set of ribs' sorts $S$ as the category, where objects are syntax diagrams with nodes signed by symbol of alphabet $A$ and ribs having sorts from the set $S$. The morphisms of the category $\catD$ are inclusion mappings of diagrams to each other. Because of inclusion mapping is associative and for each diagram there is the identical inclusion mapping of the diagram to itself, $\catD$ is really the category. It makes sense to say about one-node diagram $a$ which does not contain ribs and consists of the single node, signed the given symbol $a$. Such the diagrams are the categorical interpretation of the alphabet. There is also empty diagram, whci does not contain any node and rib. The empty diagram is included to any syntax diagram. The terminal object exists only if the alphabet $A$ consists of the single symbol $a$, then it is the one-node diagram whose node signed by symbol $a$. In the contrary case, the terminal object is ``hashed'' on one-node diagrams of the alphabet's symbols.

Obviously, it does not make sense to say about the category of all syntax diagrams above the given alphabet and the given set of ribs' sorts. The universe of the discussion is too general. It is convenient to say about the category of syntax diagrams above alphabet and ribs' sorts, that satisfy some additional conditions on the structure of nodes and ribs. For example, it is convenient to say not about all diagrams above the alphabet $A=\{a,b\}$, but only that represent chains of symbols if the language of syntax diagrams has linear representation. If $ababa$~-- the chain above alphabet $A$, then it can be represented by the diagram $a \leftarrow b \leftarrow a \leftarrow  b \leftarrow a$. The conditions are: each such diagram contains only one node, having only one rib~-- outcoming, one node having only one rib~-- incoming, and all nodes contain exactly two nodes, one incoming and one outcoming. These conditions define the subcategory of category $\catD$, containing all objects of category $\catD$, that satisfy the given conditions. Often, the conditions is the single method to define the needed set of the syntax diagrams. For example, to define the category of derivation trees in Chomsky's generative context-free grammar, it is not enough definition of the family of neighbourhood diagrams, the additional conditions should be defined. Further, $\catD$ will note the subcategory of category $\catD$ above the given alphabet and sorts of ribs, which satisfies the given conditions on view of syntax relations. In this sense, the category $\catD$ is described by using of two complementary methods: globally, by defining the conditions on the view of diagrams, and locally, by defining the neighbourhoods of the symbols.

\section{Category $\catDG$ of correct syntax diagrams}

As it is already been said above, the described formalism defines the syntax of a language by using of two methods:
\begin{enumerate}
	\item Globally~-- by enumerating the conditions on view of the multigraphs of syntax diagrams of the given language.
	\item Locally~--- by enumerating the family of neighbourhoods for each symbol of the given language's alphabet.
\end{enumerate}

Such the description may be done in several steps. At the first step, the conditions on view of the syntax diagrams are defined and so, category $\catD$ of all syntax diagrams satisfy the given conditions is described. At the second step, from the set of diagrams constructed on the first step, correct syntax diagrams are selected as the diagram satisfied by local syntax characteristics of the language. The correct syntax diagram is the object of category $\catD$, for which it exists the syntax cover of neighbourhoods of the given grammar. The syntax cover is a collection of neighbourhoods given for each diagram's node. And, if the node is signed by symbol $a$, then the element of syntax cover, which is defined for this node, should belong to the family $G_a$ of neighbourhoods of symbol $a$ of the grammar $G$. Thus, the syntax cover may be noted as the list of pairs $(v,D_a)$, where $v$~-- the node of the diagram and $D_a$~-- the neighbourhood of symbol $a$, which signs the node $v$. The category $\catDG$ of correct syntax diagrams of category $\catD$ is the category of pairs $(D,P)$, where $D \in Ob(\catD)$~-- syntax diagram and $P$~-- its syntax cover. The morphisms of category $\catDG$ are the inclusion mappings of diagrams, satisfying the condition that for each node of the subdiagram its neighbourhood, as element of the syntax cover, should be the element of syntax cover of the enveloping diagram. So that, the neighbourhoods should be identically mapped as elements of syntax covers. Obviously, this is the general case, but there may be the exceptions. It can be when there is the correct syntax diagram, which has two different syntax covers (ambiguity). To correctly say about diagrams of such the kind, it is needed to think about a pair (diagram, syntax cover) as about the single object what has been done above. Go to the formal definitions.

\begin{definition}\label{syn_cat}
	Let $G=\{G_a : a \in A\}$~-- neighbourhood grammar defined on category $\catD$. Define category $\catDG$ of correct syntax diagrams, given by the grammar $G$ on the category $\catD$, as follows:
\begin{itemize}
	\item Objects of category $\catDG$ are the pairs $(D,P)$, where $D \in Ob(\catD)$~-- syntax diagram and $P$~-- its syntax cover. $P$ is the list of pairs $(v,D_a)$, where $v$~-- node of the syntax diagram and $D_a$~-- neighbourhood of symbol $a$, which signs the node $v$, an each node $v$ is in the list $P$ exactly on one occasion.
	\item For two correct syntax diagrams $(A,P^A),(B,P^B) \in Ob(\catD)$ the set $Hom_{\catDG}((A,P^A),(B,P^B))$ consists of all inclusion mappings $s : A \rightarrow B$ such that for each node $v$ of diagram $A$ neighbourhood of node $s(v)$ in cover $P^B$ is the neighbourhood of node $v$ in cover $P^A$.
\end{itemize}
\end{definition}

It is clear the definition \ref{syn_cat} really define the category. The identity map of such category is the identity inclusion map of a correct syntax diagram to itself.

The category $\catDG$ defines correct syntax diagrams, but further we'll need the extension of this by neighbourhood diagrams. Name this new category as $Ext(\catDG)$, but often will also name it the category of correct syntax diagram.

\begin{definition}\label{ext_syn_cat}
	Let $G=\{G_a : a \in A\}$~-- neighbourhood grammar defined on category $\catD$ è $\catDG$~-- category of correct syntax diagrams, defined by the grammar $G$ on category $\catD$. Define extension $Ext(\catDG)$ of category $\catDG$ as follows:
\begin{itemize}
	\item Objects of category $Ext(\catDG)$ are the objects of category $\catDG$, and also all neighbourhood diagrams of grammar $G$.
	\item Let $A$ and $B$~-- objects of category $Ext(\catDG)$. The set of maps $Hom_{Ext(\catDG)}(A,B)$ is defined by follows:
		\begin{enumerate}
			\item If $A$ and $B$~-- correct syntax diagrams, i.e. $A,B \in Ob(Ext(\catDG))$, then $Hom_{Ext(\catDG)}(A,B) = Hom_{\catDG}(A,B)$.
			\item If $A$~-- some neighbourhood $D_a \in G_a$, and $B \in Ob(\catDG)$, then $Hom_{Ext(\catDG)}(A,B)$ consists of inclusion mappings of elements $(v,D_a) \in P$, where $P$~-- syntax cover of diagram $B$.
			\item If $A \in Ob(\catDG)$, and $B$~-- some neighbourhood $D_a \in G_a$, then $Hom_{Ext(\catDG)}(A,B) = \oslash$.
			\item If $A$ and $B$~-- some neighbourhood diagrams, then if $A$ and $B$ be the same neighbourhood $D_a \in G_a$, then $Hom_{Ext(\catDG)}(A,B) = 1_{D_a}$, where $1_{D_a}$~-- identical inclusion mapping of neighbourhood diagram $D_a$ to itself. In the contrary case $Hom_{Ext(\catDG)}(A,B) = \oslash$.
		\end{enumerate}
\end{itemize}
\end{definition}

\begin{proposition}
	$Ext(\catDG)$ is a category.
\end{proposition}
\begin{proof}
	Indeed, it is enough to show, that enumerated in the definition \ref{ext_syn_cat} morphisms satisfy the axioms of category. By definition, an each object of category $Ext(\catDG)$ has the identical map, it is the inclusion mapping of the diagram to itself. Let $D_a$ be some neighbourhood diagram. It is the object of category $Ext(\catDG)$. There exists only one map with the end on object $D_a$~-- the identical inclusion one. If $s$ is some morphism from $D_a$ to correct diagram $(D,P)$ and $s'$~-- map from diagram $(D,P)$ to correct diagram $(D',P')$, then clear, there exists composition $s' \circ s$ selecting neighbourhood $D_a$ of some node $v$ signed by symbol $a$ in diagram $D'$ by such the way that $D_a$ is the element of syntax cover $P'$. Associativity, clear, is also true. It is interesting, that if there exists some diagram of neighbourhood $D_a$, which is itself the correct diagram $(D_a,P)$ as well, then there different objects for neighbourhood diagram $D_a$ and correct diagram $(D_a,P)$ in category $Ext(\catDG)$.
\end{proof}

\section{Syntax topologies}

Topological space $(X,T)$, defined on set $X$ by topology $T$, may be viewed as the method of selection of open subsets from the set of all subsets of set $X$. An each open subset of $X$ has the cover consisting of open subsets of $X$. The same situation is in the category of syntax diagrams above the given alphabet and satisfying by the given conditions on view of diagrams. But, some neighbourhood grammar is used to select correct diagrams instead the topology. Thus, there is the idea to define a neighbourhood grammar as the topology of a special kind, defined on the category of correct syntax diagrams. But, we'll use extended category $Ext(\catDG)$ as the base for the definition.

Recall that sieve on object $A$ of category $\catC$ is the family of maps $S=\{f : Cod(f)=A\}$ satisfying by following condition: if $f \in S$ and $h : B \rightarrow Dom(f)$, then $fh \in S$. In category $Ext(\catDG)$ each morphism on some object $D$ defined either the correct subdiagram of the object $D$, or the neighbourhood diagram as element of its syntax cover. Because, sieve in category $Ext(\catDG)$ is just the set of correct and neighbourhood subdiagrams of the given diagram, closed by operation of subdiagram's getting from each its object. For example, the sieve may consists of some subdiagram of given diagram and all possible correct and neighbourhood subdiagrams of this subdiagram. These subdiagrams are also the subdiagrams of the given diagram, so the set closed and. clear, is the sieve.

There may be defined Grothendieck's topology on any small category. Recall the definition (\cite{Johnston} def. 0.32).

\begin{definition}
Grothendieck's topology on small category $\catC$ is the function $J$, which maps each object $A$ of category $\catC$ to family $J(A)$ of sieves on the object and satisfying by following conditions:
\begin{enumerate}
   \item Maximal sieve $h_A=\{f : Cod(f)=A\}$ belongs to $J(A)$.
   \item (Stability axiom) If $S \in J(A)$ and $h : B \rightarrow A$, then sieve $h^*(S)=\{f: Cod(f)=B, hf \in S\}$ belongs to $J(B)$.
   \item (Transitivity axiom) If $S \in J(A)$ and $R$~-- any sieve on $A$, such that $h^*(R) \in J(B)$ for all $B \stackrel{h}{\rightarrow} A \in S$, then $R \in J(A)$.
\end{enumerate}
Small category Grothendieck's topology with  $J$ named as site.
\end{definition}

The sieves from the families $J$ are named as $J$-covers. Obviously, if a category has fibered products, then Grothendieck's topology is defined by using of so named base~-- the families that give rise to Grothendieck's topology. In our case, it is also possible to define fibered products for objects of category $Ext(\catDG)$, but this does not make big sense in linguistic interpretation. Because the Grothendieck's topology for category $Ext(\catDG)$ will be defined by using of another definition of base~-- for categories, that have not fibered products. Such the base is defined in \cite{MacLane92} p. 156, ex. 3.

\begin{definition}\label{top_base}
Let $\textbf{C}$ be a small category. Define base of Grothendieck's topology on category $\textbf{C}$ as function $K$, which maps each object $A$ of category $\textbf{C}$ to set of morphisms' families having the end on object $A$ (covering $K$-families), satisfied by following conditions:
\begin{enumerate}
   \item If $f : B \rightarrow A$~-- isomorphism, then family $\{f : B \rightarrow A\}$ belongs to $K(A)$.
   \item (Stability axiom) If $\{f_i : A_i \rightarrow A \: |\: i \in I\} \in K(A)$, then for each morphism $g : B \rightarrow A$ there exists covering $K$-family $\{h_j : B_j \rightarrow B \: |\: j \in I'\} \in K(B)$, such that for each $j$ exists $f_i$, such that $f_i = g \circ h_j$.
   \item (Transitivity axiom) If $\{f_i : A_i \rightarrow A \: |\: i \in I\} \in K(A)$ and if for each $i \in I$ there exists family $\{g_{ij} : B_{ij} \rightarrow A_i \: |\: j \in I_i\} \in K(A_i)$, then $\{f_i \circ g_{ij} : B_{ij} \rightarrow A \: |\: i \in I, j \in I_i\} \in K(A)$.
\end{enumerate}
\end{definition}

Now define the base of Grothendieck's topology on category $Ext(\catDG)$.

\begin{definition}\label{syn_base}
Let $\catD=\{A,S,C\}$ be a category of syntax diagrams, $G=\{G_a : a \in A\}$~-- neighbourhood grammar and $Ext(\catDG)$~-- category of correct syntax diagrams. Define the base of Grothendieck's topology as function $K_G$, which maps each object $D$ of category $Ext(\catDG)$ and is satisfied by following conditions:
\begin{enumerate}
   \item Family $\{D_I\}$, where $D_I$ is isomorphism, belongs to $K_G(D)$.
   \item If $D$~-- correct diagram $(D,P)$, then family of morphisms $\{f_v : D_v \rightarrow D\}$, where $D_v \in P$for each node $v$ of diagram $D$, belongs to $K_G(D)$.
\end{enumerate}
\end{definition}

\begin{proposition}
	Function $K_G$, defined in \ref{syn_base}, is the base of some Grothendieck's topology on category $Ext(\catDG)$.
\end{proposition}

\begin{proof}
	The first axiom of base of topology definition clear is true for any object of category $Ext(\catDG)$. Show that two other axioms are also true.
	
	{\it Stability axiom.} Let $D$ be an object of category $Ext(\catDG)$, which has nontrivial covering family $K_G(D)$ (i.e. syntax cover). It's clear that a correct syntax diagram $(D,P)$ and family of inclusion maps of syntax cover $P$ on diagram $D$ is element of $K_G(D)$. If $s : D' \rightarrow D$~-- inclusion map of some diagram $D' \in Ob(Ext(\catDG))$ in diagram $D$, then there exists syntax cover $P'$ of diagram $D'$, which defines the object (the pair) $D'$ of category $Ext(\catDG)$. The inclusion maps of neighbourhood diagrams of syntax cover $P'$ are the family on $K_G(D')$. The composition of these maps with map $s$ gives exactly the needed elements of covering family $K_G(D)$.
	
	{\it Transitivity axiom.} If $K_G(D)$~-- covering family of an object $D$ of category $Ext(\catDG)$, then there is only possible covering family on each object $D$, it is trivial one (elements of syntax cover are neighbourhoods). The family of compositions of elements of trivial covers (identical maps of neighbourhoods diagrams) and inclusion mappings of each neighbourhood diagram in object $D$ is the same covering family, i.e. an element of $K_G(D)$.
\end{proof}

The function $K_G$ may be transformed to Grothendieck's topology $J$ on category $Ext(\catDG)$ by the standard way. It is enough to get all possible complements of inclusion maps of covering families of function $K_G(D)$ for each object $D \in Ext(\catDG)$. An each trivial family becames the maximal sieve on object $D$ and an each syntax cover stays the same.

\begin{definition}\label{syn_top}
Let $\catD=\{A,S,C\}$ be a category of syntax diagrams, $G=\{G_a : a \in A\}$~-- neighbourhood grammar and $Ext(\catDG)$~-- the category of correct syntax diagram on category $\catD$. Syntax topology $J_G$ based on neighbourhood grammar $G$ is the Grothendieck's topology, defined on category $Ext(\catDG)$ by the following way:
\begin{itemize}
	\item For an each object $D$ of category $\catDG$ $J_G(D)$ contains maximal sieve on object $D$.
	\item If $(D,P) \in Ext(\catDG)$~-- correct diagram, then the family of morphisms of elements of the given syntax cover $P$ belongs to $J_G(D)$.
\end{itemize}
\end{definition}

\section{Category \EtoS and sheaves, defined by syntax topologies}

It may be possible to map an each object of category $Ext(\catDG)$ to the set of some its senses (abstract meanings). The abstraction is that one does not interesting what is the concrete element of such the set, but there takes in account that this meaning exists. Even if the syntax diagram does not have any practical sense, such the set can be mapped~-- it is the empty set. The set of meanings may contain potentially infinite number of elements. Because it makes sense to use object of category $\catS$ as images of the given map.

The map of each object of category $Ext(\catDG)$ to some set of its meanings burns the contrvariant functor (name this as $F$) from category $Ext(\catDG)$ to category of sets $\catS$. Indeed, mapping $F$ is defined on each object of category $Ext(\catDG)$. For each morphism $s : D' \rightarrow D$ objects $D',D \in Ob(Ext(\catDG))$ there is map $F(f) : F(D) \rightarrow F(D')$ of accorded sets of senses, which maps each meaning $m \in F(D)$ of diagram $D$ to the meaning $m' \in F(D')$ of subdiagram $D'$, exactly the sense, which derived by subdiagram $D'$ from the diagram $D$ and meaning $m$. If $D \in Ob(Ext(\catDG))$ and $1_D$ is identical map, then the map $1_{F(D)}$ is clear defined as identical map on the set $F(D)$. It is also not hard to see that $F$ is inversely transitive functor. Thus, it is proven that $F$ is functor $F : {Ext(\catDG)}^{op} \rightarrow \catS$. The contrvariant functor from any category to category of sets is named also as subsheaf of sets. So, $F$ is subsheaf of sets on category $Ext(\catDG)$.

An each subsheaf of sets on category $Ext(\catDG)$ may be viewed as the language. So, define the category of languages defined by the given grammar $G$ as category \EtoS of contrvariant functors from $Ext(\catDG)$ to category of sets $\catS$. The objects of the category are subsheaves $F : Ext(\catDG) \rightarrow \catS$ (i.e. languages defined by that grammar) and morphisms are natural transformations of the languages. Let take a look on the properties of the category.

As it is known, (\cite{MacLane92} chapter 1) category \CtoS of subsheaf of sets on each locally small category $\catC$ (in particular, the category \EtoS) is topos, and so:
\begin{itemize}
	\item Finitely full and finitely cofull.
	\item Has exponential of any two subsheaves.
	\item Has the subobjects classifier $1 \stackrel{true}{\rightarrow} \Omega$.
\end{itemize}

Consider what is the subobjects classifier $1 \stackrel{true}{\rightarrow} \Omega$ on category \EtoS. According to (\cite{MacLane92} p. 38) subobjects classifier on category subsheaves of sets \CtoS is constructed by following way: an each object $A$ of category $\catC$ mapped to set $\Omega(A)=\{S : S \text{"--- sieve on } A\}$ of all sieves on object $A$. An each morphism $f : A \rightarrow B$ is mapped to the morphism $\Omega(f) : \Omega(B) \rightarrow \Omega(A)$, which maps each sieve $S_B \in \Omega(B)$ on object $B$ to sieve $S_A  \in \Omega(A)$ on object $A$ by getting the inverse image of morphism $f$,~-- the set $\Omega(f)(S_B)=\{h : hf \in S_B\}$. Thus, $\Omega(f)(S_B)$ selects that set of subdiagrams of diagram $A$, which are subdiagrams of diagram $B$ and belongs to $S_B$. Functor $\Omega$ classifies subfunctor $S$ of functor $F$ by the following way. Let $m \in F(B)$, there may be two cases:
\begin{enumerate}
	\item $m \in S(B)$.
	\item $m \notin S(B)$.
\end{enumerate}
The first case means that sense $m$, defined on syntax construct $B$ by functor $F$ is the same as the sense on $B$ given by its subfunctor $S$. In the case, natural transformation ${\chi}^{F}_{S}: F(B) \rightarrow \Omega(B)$ maps element $m$ to maximal sieve $h_B$ on $B$, that means the sense, given the syntax construction by functor $F$, is the same as the sense that gives the subfunctor $S$ on all subdiagrams of $B$. In the second case, there is some meaning of syntax construction $B$, given by functor $F$, which does not give by functor $S$. Then ${\chi}^{F}_{S}(B)(m)$ defines some sieve on $B$, which really is just a maximal subdiagram $A$ of diagram $B$, on which the sense that derived from the sense $m$ is equal to some meaning on $S(B)$. So, the mean of subobjects classifier $\Omega(B)$ in category \EtoS is to select the syntax subconstructions on the given diagram $B$, where the senses given by functor and subfunctor are the same. 

It is interesting what are initial and terminal objects in category \EtoS. An initial object maps an each syntax construct to empty set of senses. So, it can be seen as really formal language where any syntax construction does not have any meaning. A terminal object maps an each syntax construct to the set consisting of exactly one sense. Such the functor can be interpreted as unambiguity language. The unambiguity language amy be used to select meanings in other languages.

It is needed to define the cases the subsheaves on category $Ext(\catDG)$ are sheaves. The sheaves may be interpreted as languages that satisfy \textit{compositionality principle}, which in our terms can be formulated by the following way:

\begin{definition}
An each sense of the correct syntax diagram is uniquely defined by the senses of all its syntax subconstructions.
\end{definition}

Recall the sheaf definition on Grothendieck's topology (\cite{MacLane92} p. 122):
\begin{definition}\label{sheaf}
Let $(\catC,J)$ be a site. Subsheaf $F : {\catC}^{op} \rightarrow \catS$ is sheaf, if for each object $A \in Ob(\catC)$ and for each sieve $S \in J(A)$ diagram

$$\xymatrix{F(A) \ar[r]^(.35){e} & \prod\limits_{f \in S} F(Dom(f)) \ar@< 3pt>[r]^p \ar@<-3pt>[r]_a & \prod\limits_{f,g} F(Dom(g))}$$

where $Dom(f)=Cod(g)$, and map $e$ is equalizer of $p$ and $a$. Maps $p$ and $a$ are defined as follows:
\begin{itemize}
	\item $e={\{x \cdot f\}}_f={\{F(f)(x)\}}_f$. So that, for each $x \in F(A)$ selected the element of product $\prod\limits_{f \in S} F(Dom(f))$, consisted of images $F(f)(x)$.
	\item If ${\bf x}={\{x_f\}}_{f \in S}$~-- element of product $\prod\limits_{f \in S} F(Dom(f))$, then ${p({\bf x})}_{f,g}=x_{fg}$ and ${a({\bf x})}_{f,g}=x_f \cdot g$. So that, the map $p$ is defined via images of functor $F$ on compositions $fg$, and map $a$~-- via action of functor $F$, defined by morphism $g$ on elements $x_f$.
\end{itemize}
\end{definition}

In our interpretation, the product $\prod\limits_{f \in S} F(Dom(f))$ is just the collection of senses, selected from each subdiagram $A$, and $e$ maps each sense on diagram $A$ to collection of senses on its subdiagrams. The mapping is defined by subsheaf $F$. And also there is the condition that selected collection of senses, defined on subdiagrams, should be agreed in the sense, that if some meaning $n$ on diagram $A$ is mapped by subsheaf $F$ to meaning $m$ on subdiagram $D$, and meaning $m$ is mapped to meaning $k$ on subdiagram $D'$ of diagram $D$, then meaning $k$ uniquely mapped by functor $F$ to meaning $n$. Subsheaf $F$ is sheaf, if an each family of senses agreed on sieve $S$ has uniquely defined sense on diagram $A$ for each sieve $S \in J(A)$. As $e$ is equalizer, each sense on diagram $A$ uniquely glued from some agreed family on its subdiagrams belong to each sieve $S \in J(A)$.

In Grothendieck's topology, which is defined by some neighbourhood grammar, there are maximum two sieves on each object: maximum sieve and syntax cover for correct diagram. But, to subsheaf to be a sheave, it is enough to each family of senses agreed on syntax cover of a diagram is glued to the uniquely defined sense on this diagram. So, to understand the given subsheaf is sheaf sufficiently only to make sure that each meanings family on the neighbourhood uniquely glued to the sense on the diagram. Indeed, if $D$ is a correct syntax diagram and $D'$~-- its correct subdiagram, then each sense on $D'$ uniquely glued from senses on its neighbourhoods. So, to $F$ be a sheaf it's needed this sense should be mapped to that meaning on diagram $D$, which glued from meanings on its neighbourhoods. So, it is true for following:

\begin{proposition}\label{syn_sheaf}
Let $\catD=\{A,S,C\}$ be a category of syntax diagrams, $G=\{G_a : a \in A\}$~-- neighbourhood grammar, $Ext(\catDG)$~-- category of correct syntax diagrams on category $\catD$ and $J_G$~-- syntax topology defined by grammar $G$. Subsheaf of senses $F$ on category $Ext(\catDG)$ is sheaf if and only if each sense on correct syntax diagram $D$ is uniquely defined by each agreed family of senses on elements of its syntax cover.
\end{proposition}

It is clear that the definition of sheaf on syntax topology is exactly the reformalize of compositionality principle. And more, the definition  \ref{syn_sheaf} makes the compositionality principle ``local'', reducing its conditions to be true only on syntax covers.

The subobjects classifier $\Omega$ in category of sheaves of senses on category $Ext(\catDG)$ as usually maps an each object to set of closed in the given syntax topology sieves on the given object. Sieve $S$ on object $A$ is named as closed in Grothendieck's topology $J$, if for each morphism $f : Cod(f) = C$, if $f^*(S) \in J(Dom(f))$, then $f \in S$. So that, if set of morphisms $g$, whose compositions $fg$ with morphism $f$ belong to sieve $S$ are the covering family in $J$, then morphism $f$ also should belong to sieve $S$. In category $Ext(\catDG)$ a sieve is closed if together with each subdiagram it contains and all possible syntax covers of this subdiagram. This is the analog of principal sieve in topology on the sets. The classification of sheaf is doing obviously: an each element $x \in F(C)$ is mapped to set of morphisms $\{f | Cod(f)=C \text{ and } x \cdot f \in P(Dom(f))\}$. This set is the sieve and it is closed one if $P$ is sheaf.

Category \EtoS is an elementary topos (\cite{MacLane92} prop. 4 p. 143). In particularly, \EtoS contains all finite limit and colimits. For example, category \EtoS contains both formal and ambiguity languages, that are, accordingly, initial and terminal objects in this category.

\section{Conclusion}

Category \EtoS likes to see a convenient mathematical tool for studying both syntax and semantic properties of languages. The given in the article connection between neighbourhood grammars and Grothendieck's topologies allows to apply methods of topology and category theory to study the languages. The problems are needed to be study in future like to be follows:
\begin{itemize}
	\item Study the cases when it is possible to define neighbourhood grammars basing on an arbitrary Grothendieck's topology on site $({\bf C},J)$.
	\item Analyse the syntax complexity of languages basing on the geometrical complexity of theirs syntax diagrams.
	\item Study in more details the relations between languages defined by the given neighbourhood grammar as well as the relations between categories of sheaves of senses defined by different neighbourhood grammars.
\end{itemize}

\end{document}